\documentclass[reqno]{amsart}
\usepackage{amsfonts}%
\usepackage{amssymb}
\usepackage{mathrsfs}
\usepackage{cite}
\usepackage{color}
\allowdisplaybreaks
\date{\today}
\theoremstyle{plain}
\newtheorem{thm}{Theorem}[section]

\newtheorem{lem}[thm]{Lemma}
\newtheorem{prop}[thm]{Proposition}
\theoremstyle{definition}

\theoremstyle{remark}
\newtheorem{rem}{Remark}[section]
\numberwithin{equation}{section}

\renewcommand{\u}{{\mathbf u}}
\renewcommand{\v}{{\mathbf{v}}}

\renewcommand{\H}{\mathbf{H}}

\newcommand{\w}{{\mathbf w}}
\newcommand{\R}{{\mathbb R}}
\newcommand{\U}{{\mathbf U}}

\newcommand{\dv}{{\rm div }}
\newcommand{\cu}{{\rm curl\, }}
\newcommand{\E}{{\mathcal E}}

\newcommand{\red}{\color{red} }
\begin{document}
% \linenumbers

\title[Incompressible limit of  compressible non-isentropic MHD equations]
{Incompressible limit of the  non-isentropic ideal
  magnetohydrodynamic equations }

\author{Song Jiang}
\address{LCP, Institute of Applied Physics and Computational Mathematics, P.O.
 Box 8009, Beijing 100088, P.R. China}
 \email{jiang@iapcm.ac.cn}

\author{Qiangchang Ju}
\address{Institute of Applied Physics and Computational Mathematics, P.O.
 Box 8009-28, Beijing 100088, P.R. China}
 \email{qiangchang\_ju@yahoo.com}
 %\thanks{$^*$Corresponding author}

 \author[Fucai Li]{Fucai Li}
\address{Department of Mathematics, Nanjing University, Nanjing
 210093, P.R. China}
 \email{fli@nju.edu.cn}

\keywords{Compressible ideal MHD equations, non-isentropic,  incompressible limit}

\subjclass[2000]{76W05, 35B40}

\begin{abstract}
 We study the incompressible limit of the compressible non-
\linebreak isentropic ideal magnetohydrodynamic equations with general initial data in the whole space
$\mathbb{R}^d$ ($d=2,3$). We first establish the existence of classic solutions on a time interval independent of
the Mach number. Then, by deriving uniform a priori estimates, we
obtain the convergence of the solution to that of the incompressible
magnetohydrodynamic equations as the Mach number tends to zero.

%Furthermore, the result is valid for all nonnegative fluid viscosities
%and magnetic diffusivities.

\end{abstract}

\maketitle

\section{Introduction}

This paper is concerned with the
incompressible limit to the  compressible ideal non-isentropic magnetohydrodynamic (MHD)
equations with general initial data in the
whole space $\mathbb{R}^d$ ($d=2,3$), which take the following form
\begin{align}
&\partial_t\rho +\dv(\rho\u)=0, \label{naa} \\
&\partial_t(\rho\u)+\dv\left(\rho\u\otimes\u\right)+ {\nabla p}
  =(\nabla \times \H)\times \H , \label{nab} \\
&\partial_t\H-\nabla\times(\u\times\H)=0,\quad
\dv\H=0, \label{nac}\\
&\partial_t\E+\dv\left(\u(\E'+p)\right)
=\dv((\u\times\H)\times\H).
 \label{nad}
\end{align}
Here $\rho $ denotes the density, $\u\in \R^d$ the
velocity, and $\H\in \R^d$ the magnetic field, respectively; $\E$
is the total energy given by $\E=\E'+|\H|^2/2$ and
$\E'=\rho\left(e+|\u|^2/2 \right)$ with $e$ being the internal
energy, $\rho|\u|^2/2$ the kinetic energy, and $|\H|^2/2$ the
magnetic energy. The equations of state $p=p(\rho,\theta)$
and $e=e(\rho,\theta)$ relate the pressure $p$ and the internal
energy $e$ to the density $\rho$ and the temperature $\theta$ of the
flow.

For smooth solutions to the system \eqref{naa}--\eqref{nad},
we can rewrite the total energy equation  \eqref{nad} in the form of
the internal energy. In fact, multiplying  \eqref{nab} by $\u$ and
\eqref{nac} by $\H$ respectively, and summing the resulting equations together, we obtain
\begin{align}\label{naaz}
\frac{d}{dt}\Big(\frac{1}{2}\rho|\u|^2
+\frac{1}{2}|\H|^2\Big)
&+\frac{1}{2}\dv\big(\rho|\u|^2\u\big)+\nabla p\cdot\u \nonumber\\
& =(\nabla\times\H)\times\H\cdot\u
+\nabla\times(\u\times\H)\cdot\H.
\end{align}
Using  the identities
\begin{gather}
 \dv(\H\times(\nabla\times\H))  =|\nabla\times\H|^2-\nabla\times(\nabla\times\H)\cdot\H,\label{naeo}\\
\dv((\u\times\H)\times\H)
=(\nabla\times\H)\times\H\cdot\u+\nabla\times(\u\times\H)\cdot\H,
\label{nae}
\end{gather}
and subtracting \eqref{naaz} from \eqref{nad}, we thus obtain the
internal energy equation
\begin{equation}\label{nagg}
\partial_t (\rho e)+\dv(\rho\u e)+(\dv\u)p=0.
\end{equation}

Using the Gibbs relation
\begin{equation}\label{gibbs}
\theta \mathrm{d}S=\mathrm{d}e +
p\,\mathrm{d}\left(\frac{1}{\rho}\right),
\end{equation}
we can further replace the equation \eqref{nagg} by
\begin{equation}\label{naf}
\partial_t(\rho S)+\dv(\rho  S\u)=0,
\end{equation}
where $S$ denotes the entropy.

 Now, as in \cite{MS01}, we reconsider the equations of state as functions
of $S$ and $p$, i.e., $\rho=R(S,p)$ for
some positive smooth function $R$  defined for all $S$
and $p>0$  satisfying $\partial R/\partial p >0$. For instance,
we have $\rho=p^{1/\gamma}e^{-S/\gamma}$ with $\gamma >1$ for a polytropic gas.  Then, by
utilizing \eqref{naa} together with the constraint $\dv {\H}=0$, the
system \eqref{naa}--\eqref{nac} and \eqref{naf} can be
rewritten as
\begin{align}
    & A(S,p)(\partial_t p+(\u\cdot \nabla) p)+\dv \u=0,\label{nag}\\
& R(S,p)(\partial_t \u+(\u\cdot \nabla) \u)+\nabla p = (\nabla \times \H)\times \H, \label{nah}\\
&  \partial_t {\H} -\cu(\u\times\H)=0, \quad \dv \H=0, \label{nai}\\
 & \partial_tS+(\u\cdot \nabla) S=0,\label{naj}
\end{align}
where
\begin{align*}% \label{asp}
 A(S,p)=\frac{1}{R(S,p)}\frac{\partial R(S,p)}{\partial p}.
\end{align*}

Considering the physical explanation of the incompressible limit, we
introduce the dimensionless parameter $\epsilon$, the Mach number,
and make the following changes of variables:
\begin{gather*}%  \label{scale1}
    p (x, t)=p^\epsilon (x,\epsilon t)=\underline{p} e^{\epsilon
q^\epsilon(x,\epsilon t)}, \quad S (x, t)=S^\epsilon (x,\epsilon t), \\
    {\u} (x,t)=\epsilon \u^\epsilon(x,\epsilon t), \;\;\;
   {\H} (x,t)=\epsilon \H^\epsilon(x,\epsilon t),%\label{scale2}
\end{gather*} for some positive constant $\underline{p}$.
Under these changes of variables, the system (\ref{nag})--(\ref{naj}) becomes
\begin{align}
    & a (S^\epsilon,\epsilon q^\epsilon)(\partial_t q^\epsilon+(\u^\epsilon\cdot \nabla) q^\epsilon)
    +\frac{1}{\epsilon}\dv \u^\epsilon=0,\label{nak}\\
& r (S^\epsilon,\epsilon q^\epsilon)(\partial_t
\u^\epsilon+(\u^\epsilon\cdot \nabla)
\u^\epsilon)+\frac{1}{\epsilon}\nabla q^\epsilon
 =  ( \cu{\H^\epsilon}) \times {\H^\epsilon},  \label{nal}\\
&  \partial_t {\H}^\epsilon -\cu(\u^\epsilon\times\H^\epsilon)=0,
\quad \dv \H^\epsilon=0, \label{nam}\\
 &\partial_tS^\epsilon+(\u^\epsilon\cdot \nabla)S^\epsilon
 =0,\label{nan}
\end{align}
where we have used the abbreviations
\begin{align}
 a (S^\epsilon,\epsilon
q^\epsilon)& :=  A(S^\epsilon, \underline{p}e^{\epsilon
q^\epsilon})\underline{p}e^{\epsilon q^\epsilon}
=\frac{\underline{p}e^{\epsilon
q^\epsilon}}{R(S^\epsilon,\underline{p}e^{\epsilon
 q^\epsilon})}\cdot
 \frac{\partial R(S^\epsilon,s)}{\partial s}\Big|_{s=\underline{p}e^{\epsilon q^\epsilon}},\nonumber %\label{nann}
\\
  r (S^\epsilon,\epsilon q^\epsilon)& :=  \frac{R(S^\epsilon,\underline{p}e^{\epsilon
 q^\epsilon})}{\underline{p}e^{\epsilon q^\epsilon}}.\nonumber% \label{nano}
 \end{align}

Formally, we obtain from \eqref{nak} and \eqref{nal} that $\nabla
q^\epsilon \rightarrow 0$ and $\dv \u^\epsilon\rightarrow 0$ as $\epsilon
\rightarrow 0$. Applying the operator \emph{curl} to \eqref{nal},
using the fact that $\cu \nabla =0$,  and letting
$\epsilon\rightarrow 0$, we
obtain that
\begin{align*}
\cu\big( r(\bar{S},0)(\partial_t \v+\v\cdot \nabla \v)
 -(\cu\bar{\H}) \times \bar{\H} )=0,
\end{align*}
where we have assumed that
$(S^\epsilon,q^\epsilon,\u^\epsilon,\H^\epsilon)$ and
$r (S^\epsilon,\epsilon q^\epsilon)$ converge to
$(\bar{S},0,\v ,\bar{\H})$ and $r(\bar S,0)$ in some sense,
respectively. Finally, applying the identity
\begin{align}\label{naff}
   \cu(  { \u}\times {\H})  =
   {\u} (\dv  {\H})  -   {\H}  (\dv {\u})
 + ( {\H}\cdot \nabla) {\u} - ( {\u}\cdot \nabla) {\H},
 \end{align}
 we expect to get, as $\epsilon\to 0$, the following incompressible ideal non-isentropic MHD
 equations
\begin{align}
&   r(\bar{S},0)(\partial_t \v+(\v\cdot \nabla) \v)
  -(\cu\bar{\H}) \times \bar{\H} +\nabla \pi =0, \label{nao} \\
&  \partial_t \bar{\H}  + ( {\v}  \cdot \nabla) \bar{\H}
   - ( \bar{\H} \cdot \nabla) {\v} =0, \label{nap} \\
 &\partial_t \bar{S} +(\v \cdot \nabla) \bar{S} =0, \label{naq}\\
& \dv\,\v=0,  \quad \dv \bar{\H} =0  \label{nar}
\end{align}
for some function $\pi$.

The aim of this paper is to establish the  above limit process
rigorously in the whole space $\mathbb{R}^d$.

Before stating our main result, we briefly review the previous related
works. We begin with the results for the Euler and Navier-Stokes
equations. For well-prepared initial data, Schochet \cite{S86}
obtained the convergence of the compressible non-isentropic Euler
equations to the incompressible non-isentropic Euler equations in a
bounded domain for local smooth solutions. For general initial data,
M\'{e}tivier and Schochet \cite{MS01} proved rigorously the
incompressible limit   of the compressible non-isentropic Euler
equations  in the whole space $\R^d$. There are two key points in
the article \cite{MS01}. First, they obtained the uniform estimates in
Sobolev norms for the acoustic component of the solutions, which are
propagated by a wave equation with unknown variable coefficients.
Second, they proved that the local energy of the acoustic wave
decays to zero in the whole space case. This approach was extended
to the non-isentropic Euler equations in an exterior domain and the
full Navier-Stokes equations in the whole space by Alazard in
\cite{A05} and \cite{A06}, respectively, and to the dispersive
Navier-Stokes equations by Levermore, Sun and Trivisa \cite{LST}.
For the spatially periodic case, M\'{e}tivier and Schochet \cite{MS03} showed
the incompressible limit of the
 one-dimensional non-isentropic Euler equations with general data.
Compared to the non-isentropic case, the treatment of
the propagation of oscillations in the isentropic case is simpler
and there is an extensive literature on this topic. For example, see Ukai
\cite{U86}, Asano \cite{As87}, Desjardins and Grenier \cite{DG99} in
the whole space; Isozaki \cite{I87,I89} in an exterior domain;
Iguchi \cite{Ig97} in the half space; Schochet \cite{S94} and
Gallagher \cite{Ga01} in a periodic domain; and Lions and Masmoudi
\cite{LM98}, and Desjardins, et al. \cite{DGLM} in a bounded domain.
Recently, Jiang and Ou \cite{JO} investigated the incompressible
limit of the non-isentropic Navier-Stokes equations with zero heat
conductivity and well-prepared initial data in three-dimensional
bounded domains. The justification of the incompressible limit of the
non-isentropic Euler or Navier-Stokes equations with general initial
data in a bounded domain or a multi-dimensional periodic domain is
still open. The interested reader can refer to \cite{BDGL} on
formal computations for the case of viscous polytropic gases and
\cite{MS03,BDG} on some analysis for the non-isentropic Euler equations
 in a multi-dimensional periodic domain.
For more results on the incompressible limit of the Euler and Navier-Stokes equations,
please see the monograph \cite{FN} and the survey articles \cite{Da05,M07,S07}.

For the isentropic compressible MHD equations, the justification of
the low Mach limit has been given in several aspects. In \cite{KM},
Klainerman and Majda first studied the incompressible limit of the
isentropic compressible ideal MHD equations in the spatially periodic case with
well-prepared initial data. Recently,
 the incompressible limit of  the isentropic viscous (including both viscosity
 and magnetic diffusivity) of compressible MHD equations with
 general data was studied in \cite{HW3,JJL1,JJL2}. In
\cite{HW3}, Hu and Wang obtained the convergence of weak solutions
of the compressible viscous MHD equations in bounded,
spatially periodic domains and the whole space, respectively. In \cite{JJL1},
the authors employed the modulated energy method to verify the
limit of weak solutions of the compressible MHD equations in the
torus to the strong solution of the incompressible viscous or
partial viscous MHD equations (the shear viscosity coefficient is
zero but the magnetic diffusion coefficient is a positive constant).
In \cite{JJL2}, the authors obtained the convergence of weak
solutions of the viscous compressible MHD equations to the strong
solution of the ideal incompressible MHD equations in the whole
space by using the dispersion property of the wave equation if both
shear viscosity and magnetic diffusion coefficients go to zero.

For the full compressible MHD equations, the incompressible limit
  in the framework of the so-called variational
solutions was studied in \cite{K10,KT,NRT}. Recently, the authors \cite{JJL3}
justified rigourously the low Mach number limit of classical solutions to the
  ideal or full compressible non-isentropic MHD equations with small entropy or
  temperature variations. When the heat conductivity and large temperature variations are present,
 the low Mach number limit for the full compressible non-isentropic MHD equations
 was shown in \cite{JJLX}. We emphasize here that the arguments in \cite{JJLX}
 are completely different from the present paper (at least in the derivation of the uniform estimates),
 and depend essentially on the positivity of fluid viscosities, magnetic diffusivity, and
 heat conductivity coefficients.

As aforementioned, in this paper we want to establish rigorously the
limit as $\epsilon\to 0$ for the system \eqref{nak}--\eqref{nan}. In this case, there are additional magnetic field terms appearing in the momentum equations, compared with non-isentropic Euler equations.
When applying the non-standard derivative operators  to the momentum equations, we have to
cancel the the most-differentiated terms to close our estimates with the help of the magnetic field equations. However, it is difficult to show this cancelation due to the strong coupling of hydrodynamic motion and magnetic field. To surround such difficulties, new ideas and techniques are required in the present paper, and we
shall explain these briefly after stating our result.

Now, we supplement the system \eqref{nak}--\eqref{nan} with initial conditions
\begin{align}
(S^\epsilon,q^\epsilon,\u^\epsilon,\H^\epsilon)|_{t=0}
=(S^\epsilon_0,q^\epsilon_0,\u^\epsilon_0,\H^\epsilon_0).
\label{nas}
\end{align}
Our main result thus reads as follows.
%%%%%%%%%%%%%%%%%%%%%%%%%%%%%%%%%%%%%%%%%%%%%%
\begin{thm}\label{mth}
  Suppose that the initial data
 $S^\epsilon_0,q^\epsilon_0,\u^\epsilon_0,\H^\epsilon_0$
satisfy
\begin{align}\label{nat}
\|(S^\epsilon_0,q^\epsilon_0,\u^\epsilon_0,\H^\epsilon_0)\|_{H^4(\R^d)}\leq M_0
 \end{align}
for some constant $M_0>0$. Then there exist  constants $T>0$ and $\epsilon_0\in (0,1)$ such that for any
$\epsilon\in (0,\epsilon_0]$, the Cauchy problem \eqref{nak}--\eqref{nan}, \eqref{nas}
has a unique solution $(S^\epsilon,q^\epsilon,\u^\epsilon,\H^\epsilon)\in
C^0([0,T],H^4(\R^d))$, and there exists a positive constant $N$,
depending only on $T$, $\epsilon_0$, and $M_0$, such that
\begin{align}
 \|\big(S^\epsilon,q^\epsilon,\u^\epsilon,\H^\epsilon\big)(t)\|_{H^4(\R^d)}\leq
 N, \quad \forall\, t\in [0,T].
\label{nav}
\end{align}
Furthermore, suppose that the initial data $(S^\epsilon_0,\u^\epsilon_0,\H^\epsilon_0)$ converge in $H^4(\R^d)$ to $( S_0,
\linebreak \w_0,  \H_0)$ as $\epsilon\rightarrow 0$, and that there exist positive constants $\underline S$, $N_0$
and $\delta$, such that $S^\epsilon_0$ satisfies
\begin{equation}\label{naw}
|S^\epsilon_0(x)-\underline{S}\,\, |\leq   {N}_0 |x|^{-1-\delta},
\quad |\nabla S^\epsilon_0(x)|\leq  N_0 |x|^{-2-\delta},
\end{equation}
then  the sequence of solutions
$(S^\epsilon,q^\epsilon,\u^\epsilon,\H^\epsilon)$ converges weakly in
  $L^\infty(0,T; H^4(\R^d))$ and strongly in
$L^2(0,T;H^{s'}_{\mathrm{loc}}(\R^d))$ for all $s'<4$ to a limit
$(\bar S,0,\v,\bar\H)$, where \linebreak  $(\bar S,\v,\bar{\H})$
is the unique solution in $C([0,T],H^4(\R^d))$ of \eqref{nao}--\eqref{nar} with
initial data $(\bar S,\v,\bar{\H})|_{t=0} =( S_0,\w_0,\H_0)$,
where $\w_0\in H^4(\R^d)$ is determined by
\begin{equation}\label{nax}
   \dv\,\w_0=0,  \,\; \cu(r(S_0,0)\w_0)=\cu(r(S_0,0)\v_0),
\;\, r(S_0,0):= \lim_{\epsilon \rightarrow
0}r(S^\epsilon_0,0).
\end{equation}
The function $\pi\in C([0,T]\times \R^d)$ satisfies $\nabla\pi\in
C([0,T],H^{3}(\R^d))$.
\end{thm}

We briefly describe the strategy of the proof.
 The proof of Theorem \ref{mth} includes two main steps:
uniform estimates of the solutions, and the convergence from the
original scaling equations to the limiting ones. Once we have
established the uniform estimates \eqref{nav} in
Theorem \ref{mth}, the convergence of solutions is easily obtained by
using the local energy decay theorem for fast waves in the whole
space shown by M\'{e}tivier and Schochet in \cite{MS01}.
Thus, the main task in the present paper is to obtain the uniform
estimates \eqref{nav}. For this purpose, we shall modify the
approach developed in \cite{MS01}. In the process of
deriving uniform estimates, when we perform the operator
$(\{E \}^{-1}L(\partial_x))^\sigma$ ($\sigma=1,2,3,4$) to the continuity and momentum
equations, or the operator $\emph{curl}$ to the momentum equations,
one order higher spatial derivatives arise for the magnetic field, and the
% which have to be canceled in order to close the estimates.
 key point in the derivation of the uniform estimates is to cancel these troublesome
magnetic terms in the fluid equations by the corresponding  terms
in the magnetic field equations. Only those magnetic terms of highest order derivatives
are needed to be controlled (canceled). To achieve this cancellation, we restrict the Sobolev index $s$ to be even,
so that the highest order derivatives applied to the momentum equations are not intertwined
with the pressure equation, and we then apply the same highest order derivative operators to
the magnetic field equations. We remark that the divergence-free property of the magnetic field
plays an important role in the process of the uniform estimates.
%%%%%%%%%%%%%%%%%%%%%%%%%%%%%%%
\begin{rem}
In the proof of Theorem \ref{mth}, it is crucial to appropriately choose the index $s$ of the
 Sobolev space $H^s$ to be even (here we take $s=4$ for simplicity).
The result is still valid for any even integer $s> {d}/{2}+1$.
In contrast, for the non-isentropic Euler equations, the same result holds
in $H^s$  for any integer $s> {d}/{2}+1$ (see \cite{MS01}). The restriction on the
index is due to the strong coupling of hydrodynamic motion and magnetic field.
\end{rem}

\begin{rem}
  The uniform estimates obtained in this paper are still valid  when we consider our problem in
  a torus $\mathbb{T}^d \ (d=2,3)$ .  However, it is an open problem to prove rigorously
  the incompressible limit of the compressible non-isentropic ideal MHD equations in higher dimension $\mathbb{T}^d$ due to the loss of dispersive estimates for
  the wave equations,
  see \cite{MS03} for some
  discussions on the non-isentropic Euler equations. %Recently, the authors studied
%the incompressible limit of the one-dimensional planar compressible ideal
%magnetohydrodynamic equations \cite{JJL4}.
\end{rem}

\begin{rem} We point out that our arguments in this paper can be modified
slightly to the case of the non-isentropic MHD equations with nonnegative fluid viscosities
or magnetic diffusivity but without heat conductivity. In this case we do not require that the index $s$
of the Sobolev space $H^s$ must be even.
 The reason is that we can use the diffusion term (terms) to control the highest troublesome terms.
 For example, considering the situation that the magnetic diffusivity is zero and the
fluid viscosity coefficients are positive, we can define
 \begin{align*}
\mathcal{M}_\epsilon(T) :=\,&  { \mathcal{N}_\epsilon(T)}^2
+\int_0^T\|\u^\epsilon\|_{s+1}^2\textrm{d}\tau
\end{align*}
with
\begin{align*}
  \mathcal{N}_\epsilon(T) :=\,&  \sup_{ t\in  [0,T
]}\|(S^\epsilon,q^\epsilon,\u^\epsilon,\H^\epsilon)(t)\|_{s}, \quad s> {d}/{2}+2,
\end{align*}
and employ the same arguments with slight modifications to obtain the same result as Theorem \ref{prop}.

\end{rem}

\begin{rem}
It seems that the proof of the Theorem \ref{mth} can not be extended directly to
the full compressible MHD equations where
the heat conductivity is positive. In this situation we can not write the total energy equation
as a transport equation of entropy, which plays a key role in the proof.
Instead, it is more convenient to use the temperature as an unknown in the total energy
equation when the heat conductivity is positive, see \cite{JJLX} for more details.
\end{rem}

%%%%%%%%%%%%%%%%%%%%%%%%%%%%
\emph{Notions.}  We denote by $\langle \cdot,\cdot\rangle$ the standard inner
product in $L^2(\R^d)$ with norm $\|f\|^2=\langle f,f\rangle$ and by
 $H^k$ the usual Sobolev space $W^{k,2}(\R^d)$ with norm $\|\cdot\|_{k}$. In particular,
$\|\cdot\|_0=\|\cdot\|$. The notation $\|(A_1,  \dots, A_k)\|$ means
the summation of $\|A_i\|$ ($i=1,\cdots,k$), and it also applies to other norms.
For a multi-index $\alpha = (\alpha_1,\dots,\alpha_d)$, we denote
$\partial^\alpha =\partial^{\alpha_1}_{x_1}\dots\partial^{\alpha_d}_{x_d}$ and
$|\alpha|=|\alpha_1|+\cdots+|\alpha_d|$. We shall omit the spatial
domain $\R^d$ in integrals for convenience. We use the symbols $K$
or $C_0$ to denote generic positive constants, and $C(\cdot)$
and $\tilde{C}(\cdot)$ to denote smooth functions of their arguments which may
vary from line to line.

This paper is arranged as follows.  In Section 2, we establish the uniform boundeness of
the solutions and prove the existence part of Theorem \ref{mth}.
In Section 3, we use the decay of the local energy to the acoustic wave equations
  to prove the convergent part of Theorem \ref{mth}.

\section{Uniform estimates}

Throughout this section $\epsilon>0$, will be fixed, and
the solution $(S^\epsilon,q^\epsilon,\u^\epsilon,\H^\epsilon)$  will be denoted   by
$(S,q,\u,\H)$ and the corresponding superscript $\epsilon$ in other notations are omitted
for simplicity of presentation.

In view of \cite{MS01} and the classical local existence result in \cite{Ma} for hyperbolic systems,
we see that the key point in the proof of the existence part of Theorem \ref{mth} is to
 establish  the uniform estimate \eqref{nav}, which can be
deduced from the following \emph{a priori} estimate.
%%%%%%%%%%%%%%%%%%%%%%%%%%%%%%%%%%%%%%%%%%%%%%%%%%%%%%%%%%%%
\begin{thm}\label{prop}
Let
$(S,q,\u,\H)\in
C([0,T],H^4(\R^d))$ be a solution to \eqref{nak}--\eqref{nan}, \eqref{nas}.
Then there exist constants $C_0>0$, $0<\epsilon_0<1$ and an increasing function $C(\cdot)$ from $[0,\infty)$ to
$[0,\infty)$, such that for all
$\epsilon \in (0,\epsilon_0]$ and $t\in [0,T]$,
\begin{align}\label{nbc}
\mathcal{M}(T) \leq C_0 + (T +\epsilon )C(\mathcal{M}(T)),
 \end{align}
 where

\begin{align}
  \mathcal{M}(T) :=\,&  \sup_{ t\in  [0,T
]}\|(S,q,\u,\H)(t)\|_{4}.\label{nbbbb}
\end{align}
\end{thm}

The remainder of this section is devoted to establishing \eqref{nbc}. In the calculations
that follow, we always suppose that the assumptions in Theorem
\ref{mth} hold. We consider a solution $(S,q,\u,\H)\in C([0,T],H^4(\R^d))$ to the problem
\eqref{nak}--\eqref{nan}, \eqref{nas} with the initial data satisfying \eqref{nat}.

First, we have the following estimate of the entropy $S$, which was obtained in \cite{MS01}.

\begin{lem}  \label{NLa}
There exist a constant $C_0>0$   and a function $C(\cdot)$,
independent of $\epsilon$,  such that for all
$\epsilon \in (0,1]$ and $t\in [0,T]$,
\begin{align}\label{nbe}
\|S(t)\|^2_{4}\leq  C_0 + tC(\mathcal{M}(T)).
\end{align}
\end{lem}

The following $L^2$-bound of
$(q,\u,\H)$ can be obtained directly using the energy method due
to the skew-symmetry of the singular term in the system and the special structure of
coupling between the magnetic field and fluid velocity. This $L^2$-bound
is very important in our arguments, since it is the  starting point to the following  induction
analysis used to get the desired Sobolev estimates.

\begin{lem} \label{NLb}
  There exist a constant $C_0>0$ and an increasing function $C(\cdot)$ independent of
  $\epsilon$, such that for all $\epsilon \in (0,1]$ and $t\in [0,T]$,
\begin{align}\label{nbf}
\|(q,\u,\H)(t)\|^2 \leq  C_0 + tC(\mathcal{M}(T)).
\end{align}
\end{lem}
\begin{proof} Multiplying \eqref{nak} by $q$, \eqref{nal} by
$\u$, and \eqref{nam} by $\H$, respectively,
integrating over $\R^d$, and adding the resulting equations together, we obtain
\begin{align}
    &\langle a \partial_t q, q\rangle
+  \langle r \partial_t \u, \u\rangle
   +\langle  \partial_t \H, \H\rangle
   \nonumber\\
      &\qquad +\langle a (\u \cdot \nabla) q, q\rangle
   + \langle r (\u\cdot \nabla) \u,
   \u\rangle\nonumber\\
   & =\int[( \cu{\H}) \times
 {\H}]\cdot\u\,\textrm{d}x  +\int\cu(\u\times\H)\cdot
 \H\,\textrm{d}x.\label{first1}
  \end{align}
Here the singular terms involving $1/\epsilon$  are canceled. Using
the identity \eqref{nae} and integrating by parts, we immediately obtain
\begin{align*}
\int[( \cu{\H}) \times
 {\H}]\cdot\u\,\textrm{d}x  +\int\cu(\u\times\H)\cdot
 \H\,\textrm{d}x  =0.
\end{align*}

In view of the positivity and smoothness of $a(S,
\epsilon q)$ and $r(S, \epsilon q)$,
we get directly from \eqref{nak}, \eqref{nan} and the well-known nonlinear estimates (e.g., see \cite{Ho97}):
$$
\|f(u)\|_\sigma\leq C(\|u\|_\sigma)\|u\|_\sigma,\,\,\, u\in
H^{\sigma}(\R^d), \;\sigma>d/2,\; f(\cdot) \,\, \textrm{smooth  with}\,\,f(0)=0,
$$
 that
\begin{align}\label{nbba}
  \| \partial_t S\|_{3}\leq C(\mathcal{M}(T)), \quad
  \| \epsilon\partial_t q\|_{3}\leq C(\mathcal{M}(T)),
\end{align}
while by the Sobolev embedding theorem, we find that
\begin{align}\label{nbbb}
\|(\partial_{t}a, \partial_{t}r)\|_{L^\infty}\leq
 \|(\partial_{t}a,\partial_{t}r)\|_{2}\leq
 C(\mathcal{M}(T)).
\end{align}
By the definition of $\mathcal{M}(T)$ and the Sobolev embedding theorem,
it is easy to see that
\begin{align*}
\|(\nabla a, \nabla r)\|_{L^\infty}\leq C(\mathcal{M}(T)).
\end{align*}

Thus, from \eqref{first1} we get that
\begin{align}\label{nbg}
 \langle a   q, q\rangle
   +&  \langle r  \u, \u\rangle
    +\langle   \H, \H\rangle
\leq  \big\{ \langle a   q, q\rangle
   +  \langle r   \u, \u\rangle+\langle
   \H, \H\rangle\big\}\big|_{t=0}      \nonumber\\
  &  + C(\mathcal{M}(T))\int^t_0  \big\{|q(\tau)|^2   +|\u(\tau)|^2
 +|\H^\epsilon(\tau)|^2\big\}\textrm{d}\tau.
   \end{align}
Moreover, we have
\begin{align*}\nonumber% \label{aaaaa}
\|q\|^2+\|\u\|^2&\leq
\|(a)^{-1}\|_{L^\infty}\langle a q,
q\rangle
   + \|(r)^{-1}\|_{L^\infty}\langle r  \u,
   \u\rangle\nonumber\\
&\leq C_0(\langle a q,
q\rangle+\langle r\u, \u\rangle),
\end{align*}
since $a$ and $r$ are uniformly bounded away from
zero. Applying Gronwall's Lemma to \eqref{nbg}, we conclude that
\begin{align*}
   \|(q, \u,  \H)(t)\|^2\le C_0 \|(q_0, \u_0,
   \H_0)\|^2\exp\{tC(\mathcal{M}(T))\}.
\end{align*}
Therefore, the estimate \eqref{nbf} follows from an elementary
inequality
\begin{align}\label{ele}
e^{Ct}\leq 1+\tilde{C}t, \quad 0 \leq t \leq T_0,
\end{align}
where $T_0$ is some fixed constant.
\end{proof}

To derive the desired higher order estimates, we shall adapt and modify the
techniques developed in \cite{MS01}. Set
\begin{gather*}
  E(S,\epsilon q)=\left(\begin{array}{cc}
                   a(S,\epsilon q) & 0 \\
                    0 & r(S,\epsilon q)\mathbf{I}_{d}
                   \end{array}\right),\quad
   L(\partial_x)=\left(\begin{array}{cc}
                    0 & \dv  \\
                    \nabla & 0 \end{array}\right), \quad
    \U=\left(\begin{array}{c}
                   q \\
                   \u
                    \end{array}
                    \right),
                    \end{gather*}
where $\mathbf{I}_{d}$ denotes the  $d\times d$ unit matrix.

 Let $L_{E}(\partial_x )=  \{E(S,\epsilon q)\} ^{-1}L(\partial_x)$ and
$r_0(S) = r(S,0)$. Note that
$r_0(S)$ is smooth, positive, and bounded away from zero
with respect to each $\epsilon$. First, using Lemma \ref{NLa} and employing the same
analysis as in \cite{MS01}, we have

\begin{lem}\label{NLi}
  There exist constants $C_1>0$, $K>0$, and an increasing function $C(\cdot)$, depending
only on $M_0$, such that for all $\sigma\in\{1,2,3, 4\}$,
$\epsilon \in (0,1]$ and $t\in [0,T]$,
\begin{align}\label{nbh}
  \|\U\|_\sigma \leq K\|L(\partial_x) {\U}\|_{\sigma -1}
+\tilde{ C}\big(\|\cu
(r_0\u)\|_{\sigma-1}+\|\U\|_{\sigma-1}\big)
\end{align}
and
\begin{align}\label{nbha}
\|\U\|_\sigma \leq \tilde{
C}\big\{\|\{L_{E}(\partial_x)\}^\sigma  {\U}\|_{0}
+\|\cu(r_0\u)\|_{\sigma-1}+\|\U\|_{\sigma-1}\big\},
\end{align}
where $\tilde{C}:=C_1+tC(\mathcal{M}(T))+\epsilon C(\mathcal{M}(T)) $.
\end{lem}
We remark that the inequalities \eqref{nbh} and \eqref{nbha} are similar to the well known Helmholtz decomposition,  and the estimate on $\|S(t)\|^2_{4}$ in Lemma
\ref{NLa} plays a key role in the proof of Lemma \ref{NLi}.

Our next task is to bound $\|\{L_{E}(\partial_x)\}^\sigma  {\U}\|_{0}$ and
$\|\cu(r_0\u)\|_{\sigma-1}$ by induction arguments. We first show the following estimate.

\begin{lem}\label{NLee}
  There exists a  constant
  $C_0>0$ and an increasing function $C(\cdot)$
from $[0,\infty)$ to $[0,\infty)$, independent of $\epsilon$,  such
that for all $0\leq \sigma \leq 3$,   all
$\epsilon \in (0,1]$ and $t\in [0,T]$, it holds that
\begin{align}
 \int E|\{L_{E}(\partial_x)\}^\sigma  {\U}(t)|^2\,\textrm{\emph{d}}x
\leq C_0
+ tC(\mathcal{M}(T)),
\label{nbeea}\end{align}
and for $\sigma=4$,
\begin{align}
& \int E|\{L_{E}(\partial_x)\}^\sigma  {\U}(t)|^2\,\textrm{\emph{d}}x \nonumber\\
\leq &C_0
+ tC(\mathcal{M}(T))+2\int^t_0\int (ar)^{-4}(\nabla(\Delta^2 \H)\H)\nabla\Delta\dv\u\,\textrm{\emph{d}}x  \textrm{\emph{d}}\tau.
\label{nbeea1}\end{align}
\end{lem}

\begin{proof} Let $\U_\sigma:=\{L_{E}(\partial_x)\}^\sigma \U$,
$\sigma\in \{0,\dots, 4\}.$ For simplicity, we set $\mathcal{M}:=\mathcal{M} (T)$,
 and $E:= E(S,\epsilon q)$.
 The case $k =0$ is an immediate consequence of Lemma \ref{NLb}.
It is easy to verify that the operator $L_E(\partial_x )$ is
bounded from $H^{k}$ to $H^{k-1}$ for $k \in \{1,\dots, 5\}$. Note that
the equations \eqref{nak}, \eqref{nal} can be written as
\begin{align}\label{nbeb}
  (\partial_t+\u\cdot \nabla)\U+\frac{1}{\epsilon}E^{-1}L(\partial_x)\U=
  E^{-1}\mathbf{J}
\end{align}
with
\begin{gather*}
\mathbf{J}=\left(\begin{array}{c}
                 0 \\
                ( \cu{\H}) \times {\H}
                 \end{array}
                 \right).
\end{gather*}

 For $ \sigma \geq 1$, we commute the operator $\{L_E\}^\sigma$ with  \eqref{nbeb} and
 multiply the resulting system by $E$ to infer that
\begin{align}\label{nbeeb}
  E(\partial_t+\u\cdot \nabla)\U_\sigma
  +\frac{1}{\epsilon}L(\partial_x)\U_\sigma=E(\mathbf{f}_\sigma+
  \mathbf{g}_\sigma),
\end{align}
where
$$ \mathbf{f}_\sigma:= [\partial_t+\u\cdot \nabla , \{L_E\}^\sigma] \U,\quad
\mathbf{g}_\sigma:= \{L_E\}^\sigma (E^{-1} \mathbf{J}). $$

Multiplying \eqref{nbeeb} by $ \U_\sigma$ and integrating
over $(0,t)\times\R^d$ with $t\leq T$, noticing that
 the singular terms cancel out since $L(\partial_x)$ is skew-adjoint, we use
the inequalities \eqref{nbba} and \eqref{nbbb}, and Cauchy-Schwarz's
inequality to deduce that
\begin{align}\label{nbeec}
  \frac{1}{2}\langle E(t)\U_\sigma(t), \U_\sigma(t)\rangle
\leq & \frac{1}{2}\langle E(0)\U_\sigma(0),
\U_\sigma(0)\rangle
         + C(\mathcal{M})\int^t_0 \|\U_\sigma(\tau)\|^2 \textrm{d}\tau \nonumber\\
  & + \int^t_0 \|\mathbf{f}_\sigma(\tau)\|^2 \textrm{d}\tau+
\int^t_0\int E (\mathbf{g}_\sigma\U_\sigma)(\tau) \textrm{d}\tau.
\end{align}
For $\sigma\leq 3$, the estimate \eqref{nbeea} is obtained by
taking a similar analysis to that of Lemma 2.4 in \cite{MS01}.

Next we mainly deal with the case $\sigma=4$.
Following the proof process of Lemma 2.4 in \cite{MS01}, we obtain that
\begin{align}\label{nbeef}
 \|\mathbf{f}_4(t)\|\leq C(\mathcal{M}(t)).
\end{align}

Now we estimate the nonlinear term in \eqref{nbeec} involving
$\mathbf{g}_\sigma$. A straightforward computation implies that
\begin{align*} % \label{form}
  \U_k= \left\{
                   \begin{array}{ll}
                   \left(\begin{array}{c}
                   \{L_1L_2\}^{\frac{k-1}{2}} L_1  \u \\
                   \{L_2 L_1\}^{\frac{k-1}{2}} L_2  q
                    \end{array}
                    \right)  , &  \text{if}\  k  \ \text{is  odd}; \\[-1.0ex]\\
                    \left(\begin{array}{c}
                   \{L_1 L_2\}^{k/2} q \\
                   \{L_2 L_1\}^{k/2} \u
                    \end{array}
                    \right), & \hbox{\text{if}}\  k \  \text{is even},
                   \end{array}
                 \right.
  \end{align*}
where
\begin{align}
L_1:= a^{-1}\dv,\;\;\;\;
L_2:=r^{-1}\nabla. \nonumber
\end{align}
Therefore for $k=4$, we can replace the operator $\{L_{E}(\partial_x)\}^4$ by
\begin{align*}
\{L_{E}(\partial_x)\}^4& =\left[E^{-1}\left(\begin{array}{cc}
0&1\\
1&0
\end{array}
\right)
\left(\begin{array}{cc}
\dv&0\\
0&\nabla
\end{array}
\right)
  \right]^{4}\\
  & =
  \left(\begin{array}{cc}
(ar)^{-2}\Delta^2 &0\\
0 &(ar)^{-2}(\nabla\dv)^{2}
\end{array}
\right)
 +\mathfrak{L},
\end{align*}
where $\mathfrak{L}$ is a matrix whose elements are the
summation of the lower order differential operators up to order $3$.
Thus by virtue of Cauchy-Schwarz's and Sobolev's inequalities, and \eqref{nbbb},
we infer that
\begin{align*}
\int^t_0\int E\mathbf{g}_4\U_4\textrm{d}x\textrm{d}\tau
=&\int^t_0\int E(L_{E})^4(E^{-1}\mathbf{J}) \U_4\textrm{d}x\textrm{d}\tau\\
\leq &\int^t_0\int (ar)^{-4}[(\nabla\dv)^2(\cu{\H} \times {\H})]
(\nabla\dv)^{2}\u\,\textrm{d}x  \textrm{d}\tau+C(\mathcal{M})\\
=: & I_1(t)+C(\mathcal{M}).
\end{align*}
%%%
To control $I_1(t)$, we recall the basic vector identities:
$$ \dv (\mathbf{a}\times \mathbf{b}) =\mathbf{b}\cdot \cu \mathbf{a} -\mathbf{a}\cdot \cu \mathbf{b},
\quad \cu \cu \mathbf{a}=\nabla \,\dv \,\mathbf{a} -\Delta \mathbf{a}, $$
 and the fact that $\dv \H=0$ to deduce
  \begin{align*}
I_1(t) &\leq \int^t_0\int (ar)^{-4}(\nabla(\Delta^2 \H)\H)\nabla\Delta\dv\u\,\textrm{d}x
\textrm{d}\tau+C(\mathcal{M}). \end{align*}
Thus, we have
\begin{align}
&\int^t_0\int E\mathbf{g}_4\U_4\, \textrm{d}x\textrm{d}\tau
\leq  \int^t_0\int (ar)^{-4}(\nabla(\Delta^2 \H)\H)\nabla\Delta\dv\u\,
\textrm{d}x \textrm{d}\tau+C(\mathcal{M})\label{new1}.
\end{align}

Finally, \eqref{nbeea1} follows from the above estimates \eqref{nbeec}--\eqref{new1} and the
positivity of $E$.
\end{proof}

Next, we derive an estimate for
$\|\cu(r_0\u)\|_{\sigma-1}$. Define
\begin{align}\label{factor}
 f(S,\epsilon q): =1-\frac{r_0(S)}{r(S, \epsilon q)}.
\end{align}
Hereafter we denote $r_0(t) := r_0(S(t))$ and
$f(t):= f(S(t), \epsilon q(t))$
for notational simplicity.

One can factor out $\epsilon q$ in $f(t)$. In fact,
using Taylor's expansion, one obtains that there exists a smooth
function $g(t)$, such that
\begin{align}\label{nbfe}
f(t)=\epsilon g(t):= \epsilon g(S(t), \epsilon q(t)), \quad
\|g(t)\|_4\leq  C(\mathcal{M}(T)).
\end{align}

Since
$$\partial_t S + (\u\cdot\nabla)S = 0, $$
 the momentum equations  \eqref{nal} are equivalent to
\begin{align}\label{nbff}
[\partial_t + (\u\cdot \nabla)] (r_0\u)
+\frac{1}{\epsilon}\nabla q  =\, g\nabla q
+ (1-\epsilon g)( \cu{\H}) \times {\H}.
\end{align}
%%%%%
We perform the operator \emph{curl} to the system \eqref{nbff} to obtain that
\begin{align}
& [\partial_t +(\u\cdot \nabla)](\cu(r_0\u))\nonumber\\
 = &\, [\u\cdot \nabla, \cu](r_0\u) + [\cu,g]\nabla q +
 \cu [(1-\epsilon g)( \cu{\H}) \times {\H}],  \label{nbfg}
 \end{align}
 and we can show
%%%%%%%%%%%%%%%%%%%%%%%%%%%%%%%%%%%%%%%%%
\begin{lem}\label{NLg}
  There exist constants $C_0>0$, $0<\epsilon_0<1$  and an increasing function $C(\cdot)$
from $[0,\infty)$ to $[0,\infty)$, such that for all $\epsilon \in (0,\epsilon_0]$ and $t\in [0,T]$, it holds that
\begin{align}  \label{nbga1}
\|\{\cu (r_0\u)\|^2_{2} \leq  C_0 + tC(\mathcal{M}(T)),
\end{align}
and for $|\alpha|=3$,
\begin{align}  \label{nbga2}
& \sum_{|\alpha|=3} \langle \partial^\alpha \cu (r_0\u)(t), F \partial^\alpha \cu (r_0\u)(t)\rangle
\leq C_0 + tC(\mathcal{M}(T))\nonumber\\
&\qquad \qquad  +2 \sum_{|\alpha|=3} \int^t_0\int (ar)^{-4}(\H\cdot\nabla)\partial^\alpha
\H\partial^\alpha(\Delta\u)\,\textrm{\emph{d}}x\textrm{\emph{d}}\tau,
\end{align}
where we denote by $F:=(ar)^{-4}(r_0(1-\epsilon g))^{-1}>0$.
\end{lem}

\begin{proof}
The estimate \eqref{nbga1} directly follows from an energy estimate performed on \eqref{nbfg}, and hence
we omit the details here. Here we give a proof of \eqref{nbga2} only.

Set  $\mathcal{M}:=\mathcal{M}(T)$,
and  $\omega= \cu (r_0\u)$.  By the defination of $f(S,\epsilon q)$ and \eqref{nbfe},
there exists an $\epsilon_0\in (0,1]$, such that
$1-\epsilon g$ is bounded from below for $\epsilon \in (0,\epsilon_0]$. Taking $\partial^\alpha_x$
$(|\alpha|= 3)$ to \eqref{nbfg}, multiplying
the resulting equations by $F\partial^\alpha_x\omega$, and
integrating over $(0,t)\times \R^d$ with $t\leq T$, we obtain that
\begin{align}\label{nbgb}
\frac{1}{2}\langle \partial^\alpha \omega(t), F \partial^\alpha \omega(t)\rangle
\leq\,&\frac{1}{2}\langle \partial^\alpha \omega(0), F \partial^\alpha \omega(0)\rangle
+C(\mathcal{M})\int_0^t\|\partial^\alpha \omega(\tau)\|^2\textrm{d}\tau\nonumber \\
&-\int^t_0\langle(\u^\epsilon\cdot \nabla)
\partial^\alpha  \omega, F\partial^\alpha  \omega\rangle(\tau)\textrm{d}\tau\nonumber\\
&+\int^t_0\langle [\u\cdot \nabla, \partial_x^\alpha]\omega, F\partial^\alpha
 \omega\rangle(\tau) \textrm{d}\tau\nonumber\\
& +\int^t_0\langle \partial^\alpha \{ [\cu, g^\epsilon]\nabla q\},
F\partial^\alpha  \omega\rangle(\tau) \textrm{d}\tau\nonumber\\
& +\int^t_0 \langle \partial^\alpha \{[\u^\epsilon\cdot \nabla, \cu](r_0\u)\},
  F\partial^\alpha  \omega\rangle(\tau) \textrm{d}\tau\nonumber\\
&+\int^t_0 \langle \partial^\alpha \{\cu [(1-\epsilon g)( \cu{\H})
\times {\H}]\}, F\partial^\alpha  \omega\rangle(\tau) \textrm{d}\tau \nonumber\\
=&:\frac{1}{2}\langle \partial^\alpha \omega(0), F \partial^\alpha \omega(0)\rangle+ \int^t_0\sum_{i=1}^6 I_i(\tau).
 \end{align}
We have to estimate the terms $I_i(\tau)$ ($2\leq i\leq 6$) on the right-hand side of (\ref{nbgb}).
Applying  integrations by parts, we have
\begin{align}
I_2(\tau)&\leq \int_{\mathbb{R}^d}  F|   \partial^\alpha  \omega  |^2 \dv\u\, \textrm{d}x+C(\mathcal{M})\nonumber\\
  & \leq \|\dv \u(\tau)\|_{L^\infty}\|\partial^\alpha   \omega(\tau)\|^2+C(\mathcal{M})
  \leq C(\mathcal{M}),
    \label{nbgc}
\end{align}
while for the term $I_3(\tau)$, an application of Cauchy-Schwarz's inequality gives
\begin{align*}
 | I_3(\tau)|\leq   C\|F\partial^\alpha  \omega\| \, \|\mathbf{h}_\alpha(\tau)\|,
 \quad  \mathbf{h}_\alpha(\tau):= [\u\cdot \nabla, \partial_x^\alpha]\omega.
\end{align*}
The commutator $\mathbf{h}_\alpha$ is a sum of terms
$\partial^\beta_x \u\partial^\gamma_x \omega$
with multi-indices $\beta$ and $\gamma$ satisfying $|\beta|+|\gamma|\leq 4$, $|\beta|>0$,
and $|\gamma|>0$. Thus,
$$\|\mathbf{h}_\alpha(\tau)\|\leq C(\mathcal{M}),$$
 where the the following nonlinear Sobolev inequality has been used (see \cite{Ho97}): For all
$\alpha=(\alpha_1, \cdots, \alpha_d)$, $\sigma\geq 0$, and $f,g\in H^{k+\sigma}(\R^d)$, $|\alpha|=k$,
it holds that
$$ \|[f,\partial^\alpha ]g\|_{{\sigma}}\leq  C_0(\|f\|_{W^{1,\infty}}\| g\|_{{\sigma+k-1}}
    +\| f\|_{{\sigma+k}}\|g\|_{L^{\infty}}).
$$
Hence, we have
 \begin{align}\label{nbgd}
 | I_3(\tau)|\leq  C(\mathcal{M})+\|\partial^\alpha  \omega(\tau)\|^2.
\end{align}

Noting that $([\cu, g]\,\mathbf{a}\,)_{i,j} =
a_i\partial_{x_j} g - a_j\partial_{x_i}g$ for $\mathbf{a}=(a_1,\cdots, a_d)$,
using the  estimate \eqref{nbfe}, and the inequality
$\|uv\|_{\sigma-k-l}\leq K\|u\|_{\sigma-k}\|v\|_{\sigma-l}$ for $ k\geq 0$, $l\geq 0$,
$k+l\leq\sigma$ and $\sigma > d/2$ (see \cite{Ho97}),
we can control the term $I_4(\tau)$ as follows.
 \begin{align}\label{nbge}
    |I_4(\tau)|& \leq \|\partial^\alpha \{ [\cu, g]\nabla q\}\|\; \|
    F\partial^\alpha  \omega\|\nonumber\\
     & \leq K \| [\cu, g]\nabla q\|_{3}\, \| F\partial^\alpha  \omega\|\nonumber\\
    & \leq K \|\nabla g(\tau)\|_{3} \|\nabla q(\tau)\|_{3}\, \|
    F\partial^\alpha  \omega\|\nonumber\\
    &\leq  C(\mathcal{M})+\|\partial^\alpha  \omega(\tau)\|^2.
 \end{align}

Similarly, the term $I_5(\tau)$ can be bounded as follows.
 \begin{align}\label{nbgf}
    |I_5(\tau)|& \leq K\|\partial^\alpha \{[\u\cdot \nabla, \cu](r_0\u)\}\|
    \; \| F\partial^\alpha  \omega\|\nonumber\\
& \leq K\|[\u\cdot \nabla, \cu](r_0\u^\epsilon)\|_{3}\, \| F\partial^\alpha  \omega\|\nonumber\\
      & \leq K\|[\u_j, \cu]\partial_{x_j}(r_0\u)\|_{3}\,
      \| F\partial^\alpha  \omega\|\nonumber\\
        &\leq  C(\mathcal{M})+\|\partial^\alpha  \omega(\tau)\|^2.
 \end{align}

For the term $I_6(\tau)$, by virtue of  the formula
$\cu \cu \mathbf{a} =\nabla \,\dv \,\mathbf{a} -\Delta \mathbf{a}$, we
 integrate by parts to see that
\begin{align*}
   I_6(\tau)  & = \Big\langle \frac{(ar)^{-4}}{r_0(1-\epsilon g)}\partial^\alpha \cu\Big\{(1-\epsilon g)
((\H\cdot\nabla)\H-\nabla(\frac{|\H|^2}{2})\Big\},\partial^\alpha\cu( r_0\u)\Big\rangle
\nonumber\\
&\leq - \Big\langle \frac{(ar)^{-4}}{r_0(1-\epsilon g)}\partial^\alpha\{(1-\epsilon g)
\H\cdot\nabla)\H\},\partial^\alpha\cu\cu( r_0\u)\Big\rangle+C(\mathcal{M})\\
&\leq -\langle(ar)^{-4}\partial^\alpha[(\H\cdot\nabla)]\H,\partial^\alpha(\nabla\dv\u-\Delta\u)\rangle+C(\mathcal{M})\\
\red &\leq \langle(ar)^{-4}(\H\cdot\nabla)\partial^\alpha\H,\partial^\alpha\Delta\u\rangle+C(\mathcal{M}),
\end{align*}
  where we have used the fact that $\dv\H=0$.

Inserting the estimates for $I_j(\tau)$ ($2\leq j\leq 6$) into \eqref{nbgb}
and summing over $|\alpha|=3$, we obtain \eqref{nbga2}. The lemma is proved.
\end{proof}

We proceed to estimating the magnetic field.
\begin{lem}\label{lem7}
  There exist a constant $C_0>0$ and an increasing function $C(\cdot)$
from $[0,\infty)$ to $[0,\infty)$, such that
for $|\alpha|\leq 2$, all $\epsilon \in (0,1]$ and $t\in [0,T]$, we have
\begin{align}
 \sum_{|\alpha|\leq 2}\int (ar)^{-4}| \partial^\alpha\cu\H|^2\,\textrm{\emph{d}}x
\leq C_0+tC(\mathcal{M}(T))\label{al1}
\end{align}
and for $|\alpha|=3$,
\begin{align}
& \sum_{|\alpha|=3} \int (ar)^{-4}| \partial^\alpha\cu\H|^2\,\textrm{\emph{d}}x   \nonumber\\
\leq & C_0+tC(\mathcal{M}(T))
 -2\sum_{|\alpha|=3}\int_0^t\int (ar)^{-4}(\H\cdot\nabla)\partial^{\alpha}\H\partial^\alpha (\Delta\u)\,\textrm{\emph{d}}x\textrm{\emph{d}}\tau\nonumber\\
  & -2\int_0^t\int (ar)^{-4}(\nabla\Delta^2 \H\H)\nabla\Delta\dv\u\,\textrm{\emph{d}}x\textrm{\emph{d}}\tau. \label{nbgaz}
\end{align}
\end{lem}
\begin{proof} Set $\mathcal{M}:=\mathcal{M}(T)$. We apply the operator \emph{curl} to
\eqref{nam} and use the vector identity \eqref{naff} to obtain that
\begin{equation}\label{nbgh}
  \partial_t (\cu \H)+\u\cdot\nabla(\cu\H)  = -[\cu, \u]\cdot\nabla
\H+\cu((\H\cdot\nabla)\u-\H\dv\u ). \end{equation}
Taking $\partial^\alpha(|\alpha|\leq 3)$ to \eqref{nbgh},
and multiplying the resulting equations by $(ar)^{-4} \times \partial^\alpha\cu\H$,
and integrating over $(0,t)\times \R^d$ with $t\leq T$, we find that
\begin{align}
&\langle (ar)^{-4}\partial^\alpha\cu\H, \partial^\alpha\cu\H\rangle\nonumber\\
=&
 \big\{\big\langle \{a(t)r(t)\}^{-4}\partial^\alpha\cu\H(t), \partial^\alpha\cu\H(t)\big\rangle\big\}\big|_{t=0}\nonumber\\
  & +\langle \partial_t\{(ar)^{-4}\}\partial^\alpha\cu\H, \partial^\alpha\cu\H\rangle\nonumber\\
 & -2\int_0^t\int (ar)^{-4}\partial^\alpha[(\u\cdot\nabla)\cu\H)\partial^\alpha\cu\H]\,\textrm{d}x
 \textrm{d}\tau\nonumber\\
  & -2\int_0^t\int (ar)^{-4}\partial^\alpha\{[\cu, \u]\cdot\nabla\H\}(\partial^\alpha\cu\H)\,\textrm{d}x
 \textrm{d}\tau\nonumber\\
&+2\int_0^t\int (ar)^{-4}\partial^\alpha\cu((\H\cdot\nabla)\u-\H\dv\u)(\partial^\alpha\cu\H)\,\textrm{d}x
 \textrm{d}\tau.\nonumber
\end{align}

For $|\alpha|\leq 2$, it is easy to see that
\begin{align}\label{nbga}
\sum_{|\alpha|\leq 2}\int(ar)^{-4}|\partial^\alpha\cu \H|^2\;\textrm{d}x
\leq  C_0 + tC(\mathcal{M}).
\end{align}

Next, we treat the case $|\alpha|=3$. By a straightforward calculation we arrive at
\begin{align}
&\langle (ar)^{-4}\partial^\alpha\cu\H, \partial^\alpha\cu\H\rangle\nonumber\\
\leq&  C_0 + tC(\mathcal{M})+2\int_0^t\int (ar)^{-4}\partial^\alpha\cu((\H\cdot\nabla)\u
-\H\dv\u)(\partial^\alpha\cu\H)\,\textrm{d}x  \textrm{d}\tau\nonumber\\
  =:&C_0 + tC(\mathcal{M})+2\int^t_0 I_\alpha(\tau)\textrm{d}\tau.\nonumber
 \end{align}
For the term $I_\alpha(\tau)$, using the fact that $\dv \H=0$, we have
\begin{align*}
\sum_{|\alpha|=3}I_\alpha(\tau)=&\sum_{|\alpha|=3}\int (ar)^{-4}\partial^{\alpha}
 \cu( (\H\cdot\nabla)\u - \H\dv\u )\partial^{\alpha}\cu\H\,\textrm{d}x  \\
\leq&-\sum_{|\alpha|=3}\int\big\{ (ar)^{-4}\partial^{\alpha}( -(\u\cdot\nabla)\H
+ (\H\cdot\nabla)\u - \H\dv\u )\nonumber\\
& \times \partial^{\alpha}\cu\cu\H\big\}\,\textrm{d}x  +C(\mathcal{M})\\
\leq &\sum_{|\alpha|=3}\int (ar)^{-4}\partial^{\alpha}(  (\H\cdot\nabla)\u
- \H\dv\u )\partial^{\alpha}\Delta\H\,\textrm{d}x  +C(\mathcal{M})\\
\leq&-\sum_{|\alpha|=3}\int (ar)^{-4}(\H\cdot\nabla)\partial^{\alpha}
\H\partial^\alpha(\Delta\u)\,\textrm{d}x \nonumber\\
&  -\int (ar)^{-4}(\nabla\Delta^2 \H\H)\nabla\Delta\dv\u\,\textrm{d}x  +C(\mathcal{M}),
\end{align*}
where we have used the integration by parts several times in the last step. Thus the lemma is proved,
taking into account the positivity of $a(S,\epsilon q)$ and $r(S,\epsilon q)$.
\end{proof}

Finally, puting Lemmas \eqref{NLb}--\eqref{lem7} together,
and using the induction argument as in \cite{MS01},
we get the following estimate which finishes the proof of Theorem \ref{prop}.
%%%%%%%%%%%%%%%%%%%%%%%%%%%%%%%%%%%%
\begin{lem}\label{NLk}
There exist a constant $C_0>0$ and an increasing function
$C(\cdot)$ from $[0,\infty)$ to $[0,\infty)$, such that for all
$\epsilon \in (0,\epsilon_0]$ and $t\in [0,T]$,
\begin{align}\label{nbja}
\|(S,q,\u,\H)(t)\|^2_{s}\leq C_0+(t+\epsilon)C(\mathcal{M}(T)).
\end{align}
\end{lem}

\section{Incompressible limit}

In this section, we prove the convergence part of Theorem \ref{mth}
by modifying the method developed by M\'{e}tivier and Schochet
\cite{MS01}, see also some extensions in \cite{A05,A06, LST}.

\begin{proof}[Proof of the convergence part of Theorem \ref{mth}]
 The uniform bound \eqref{nav} implies, after extracting a subsequence, the following limit:
   \begin{align}
      & (q^\epsilon,\u^\epsilon,\H^\epsilon )\rightharpoonup (q,\v,\bar\H )\quad \text{weakly-}\!\ast
     \  \text{in} \quad L^\infty (0,T; H^4(\mathbb{R}^d)).
\label{caa}
\end{align}

 The equations \eqref{nam} and \eqref{nan} imply that
 $\partial_tS^\epsilon$ and $\partial_t \H^\epsilon \in C([0,T],H^{3}(\mathbb{R}^d))$.
 Thus, after further extracting a subsequence, we
  obtain that, for all $s'<4$,
\begin{align}
 S^\epsilon \rightarrow \bar S &  \quad \text{strongly in}  \quad
 C([0,T],H^{s'}_{\mathrm{loc}}(\mathbb{R}^d)),\label{cab}\\
\H^\epsilon  \rightarrow \bar \H &\quad \text{strongly in} \quad
C([0,T],H^{s'}_{\mathrm{loc}}(\mathbb{R}^d)),\label{cad}
\end{align}
where the limit $\bar \H \in C([0,T],H^{s'}_{\mathrm{loc}}(\mathbb{R}^d))\cap
L^\infty (0,T;H^{4}_{\mathrm{loc}}(\mathbb{R}^d))$.
Similarly, by \eqref{nbfg} and the uniform bound \eqref{nav}, we have
\begin{align}
 & \cu (r_0(S^\epsilon) \u^\epsilon) \rightarrow \cu (r_0(\bar S) \v) \quad
 \text{strongly in}  \quad C([0,T],H^{s'-1}_{\mathrm{loc}}(\mathbb{R}^d))\label{cac}
  \end{align}
for all $ s'<4$, where $ r_0(\bar S)=\lim_{\epsilon \rightarrow
0}r_0(S^\epsilon):= \lim_{\epsilon \rightarrow 0}r(S^\epsilon,0)$.

In order to obtain the limit system, we need to prove that the
convergence in \eqref{caa} holds in the strong topology of
$L^2(0,T;H^{s'}_{\mathrm{loc}}(\mathbb{R}^d))$ for all $s'<4$. To this end,
we first show that $q=0$ and $\dv \v=0$. In fact, from \eqref{nbeb} we get
\begin{align}\label{nbeb1}
 \epsilon E(S^\epsilon,\epsilon q^\epsilon)\partial_t\U^\epsilon
 +L(\partial_x)\U^\epsilon=-\epsilon E (S^\epsilon,\epsilon q^\epsilon)
 \u^\epsilon\cdot \nabla\U^\epsilon+\epsilon \mathbf{J}^\epsilon.
\end{align}
Since
\begin{align*}
 E (S^\epsilon,\epsilon
 q^\epsilon)-E (S^\epsilon,0)=O(\epsilon),
\end{align*}
we have
\begin{align}\label{pass}
\epsilon E (S^\epsilon,0)\partial_t\U^\epsilon
 +L(\partial_x)\U^\epsilon=\epsilon \mathbf{h}^\epsilon ,
\end{align}
where, by virtue of \eqref{nav}, $\mathbf{h}^\epsilon$ is uniformly bounded
in $C([0,T],H^{3}(\mathbb{R}^d))$. Passing to the weak
limit in \eqref{pass}, we obtain $\nabla q=0$ and $\dv \v=0$.
Since $q\in L^\infty(0,T;H^4(\mathbb{R}^d))$, we infer that $q=0$.
Now to complete the proof of Theorem \ref{mth},
we first prove the following proposition.

\begin{prop}\label{LC}
Suppose that the assumptions in Theorem \ref{mth} hold, then
$q^\epsilon$ converges strongly to $0$ in
$L^2(0,T;H^{s'}_{\mathrm{loc}}(\mathbb{R}^d))$ for all $s'<4$, and $\dv
\u^\epsilon$ converges strongly to $0$ in $L^2(0,T;
H^{s'-1}_{\mathrm{loc}}(\mathbb{R}^d))$ for all $s'<4$.
\end{prop}

The proof of Proposition \ref{LC} is built on the the following
dispersive estimates on the wave equations established by M\'{e}tivier
and Schochet \cite{MS01} and reformulated in \cite{A06}.

\begin{lem}[\!\!\cite{MS01,A06}]\label{LD}
   Let $T>0$ and $w^\epsilon$ be a bounded sequence in $C([0,T],\linebreak H^2(\mathbb{R}^d))$, such that
   \begin{align*}
    \epsilon^2\partial_t(b^\epsilon \partial_t w^\epsilon)-\nabla\cdot (c^\epsilon \nabla w^\epsilon)=e^\epsilon,
   \end{align*}
where $e^\epsilon$ converges to $0$ strongly in $L^2(0,T;
L^2(\mathbb{R}^d))$. Assume in addition that, for some $s> d/2+1$, the
coefficients $(b^\epsilon, c^\epsilon)$ are uniformly bounded in
$C([0,T],H^4(\mathbb{R}^d))$ and converges in $C([0,T],H^4_{\mathrm{loc}}(\mathbb{R}^d))$
to a limit $(b,c)$ satisfying the decay estimates
\begin{gather*}
  |b(x,t)-\underline b|\leq C_0 |x|^{-1-\delta}, \quad |\nabla_x b(x,t)|\leq C_0 |x|^{-2-\delta}, \\
 |c(x,t)-\underline c|\leq C_0 |x|^{-1-\delta}, \quad |\nabla_x c(x,t)|\leq C_0 |x|^{-2-\delta},
\end{gather*}
for some given positive constants $\underline b$, $\underline c$, $
C_0$ and $\delta$. Then the sequence $w^\epsilon$ converges to $0$
in  $L^2(0,T; L^2_{\mathrm{loc}}(\mathbb{R}^d))$.
\end{lem}

\begin{proof}[Proof of Proposition \ref{LC}]
We first show that $q^\epsilon$ converges strongly to $0$ in
$L^2(0,T; \linebreak H^{s'}_{\mathrm{loc}}(\mathbb{R}^d))$ for all
$s'<4$. An application of the operator $\epsilon^2 \partial_t$ to \eqref{nak}
gives
\begin{align}\label{cae}
\epsilon^2\partial_t (a (S^\epsilon,\epsilon q^\epsilon)
\partial_t q^\epsilon) +\epsilon \partial_t\dv \u^\epsilon
=-\epsilon^2 \partial_t\{a (S^\epsilon,\epsilon
q^\epsilon)(\u^\epsilon\cdot\nabla )q^\epsilon\}.
\end{align}
Dividing \eqref{nal} by $r^\epsilon(S^\epsilon,\epsilon q^\epsilon)$
and then taking the operator \emph{div} to the resulting equations, one has
\begin{align}\label{caf}
& \partial_t \dv \u^\epsilon +\frac{1}{\epsilon}\dv
\Big(\frac{1}{r (S^\epsilon,\epsilon q^\epsilon)}\nabla q^\epsilon\Big)\nonumber\\
& \qquad \qquad =-\dv ((\u^\epsilon\cdot \nabla)\u^\epsilon)+\dv
\Big(\frac{1}{r (S^\epsilon,\epsilon q^\epsilon)}(\cu
\H^\epsilon)\times \H^\epsilon\Big).
\end{align}
Subtracting \eqref{caf} from \eqref{cae}, we get
\begin{align}\label{cag}
\epsilon^2\partial_t (a (S^\epsilon,\epsilon q^\epsilon)
\partial_t q^\epsilon) -\dv
\Big(\frac{1}{r (S^\epsilon,\epsilon q^\epsilon)}\nabla
q^\epsilon\Big) =  F (S^\epsilon, q^\epsilon, \u^\epsilon,
\H^\epsilon),
\end{align}
where
\begin{align*}
 F (S^\epsilon, q^\epsilon, \u^\epsilon, \H^\epsilon)
  =\, & \epsilon  \dv \Big(\frac{1}{r (S^\epsilon,\epsilon q^\epsilon)}
  (\cu \H^\epsilon)\times \H^\epsilon\Big)\\
& -\epsilon \dv
((\u^\epsilon\cdot \nabla)\u^\epsilon) - \epsilon^2
\partial_t\{a (S^\epsilon,\epsilon
q^\epsilon)(\u^\epsilon\cdot\nabla )q^\epsilon\}.
\end{align*}

In view of the uniform boundedness of $(S^\epsilon,q^\epsilon,\u^\epsilon,\H^\epsilon)$,
the smoothness and positivity assumptions on $a(S^\epsilon,\epsilon q^\epsilon)$
and $r (S^\epsilon,\epsilon q^\epsilon)$, and the convergence of $S^\epsilon$, we find that
 %%%%%%%%%%%%%%%%%%%%%%
\begin{align*}
F (S^\epsilon, q^\epsilon, \u^\epsilon, \H^\epsilon)
\rightarrow 0 \quad \text{strongly in} \quad L^2(0,T;
L^2(\mathbb{R}^d)),
\end{align*}
and the coefficients in \eqref{cag} satisfy the requirements in
Lemma \ref{LD}. Therefore, by virtue of Lemma \ref{LD},
\begin{align*}
q^\epsilon \rightarrow 0 \quad \text{strongly in} \quad L^2(0,T;
L^2_{\mathrm{loc}}(\mathbb{R}^d)).
\end{align*}

On the other hand, the uniform boundedness of $q^\epsilon$ in $C([0,T],H^4(\mathbb{R}^d))$
and an interpolation argument yield that
\begin{align*}
q^\epsilon \rightarrow 0 \quad \text{strongly in}
 \quad    L^2(0,T;  H^{s'}_{\mathrm{loc}}(\mathbb{R}^d))\ \ \text{for all} \ \  s'<4.
\end{align*}
Similarly, we can obtain the  convergence of  $\dv u^\epsilon$.
\end{proof}

We continue our proof of Theorem \ref{mth}. From Proposition
\ref{LC}, we know that
\begin{align*}
\dv\, \u^\epsilon \rightarrow \dv\,\v\quad \mathrm{in} \quad L^2(0,T;
H^{s'-1}_{\mathrm{loc}}(\mathbb{R}^d)).
\end{align*}
Hence, from \eqref{cac} it follows that
\begin{align*}
 \u^\epsilon \rightarrow \v\quad \mathrm{in} \quad L^2(0,T;
H^{s'}_{\mathrm{loc}}(\mathbb{R}^d))\qquad\mbox{for all }s'<4.
\end{align*}
By \eqref{cab}, \eqref{cad} and Proposition \ref{LD}, we obtain
\begin{equation*}
\begin{array}{ccl}
r^\epsilon(S^\epsilon, \epsilon q^\epsilon) \rightarrow r_0(\bar S)
& \mathrm{in} &  L^\infty(0,T; L^\infty(\mathbb{R}^d));\\
\nabla \u^\epsilon\rightarrow\nabla\v & \mathrm{in}
& L^2(0,T; H^{s'-1}_{\mathrm{loc}}(\mathbb{R}^d));\\
\nabla \H^\epsilon\rightarrow \nabla  \bar \H    & \mathrm{in} &
L^2(0,T; H^{s'-1}_{\mathrm{loc}}(\mathbb{R}^d)).
\end{array}
\end{equation*}
Passing to the limit in the equations for $S^\epsilon$ and
$\H^\epsilon$, we see that the limits $\bar{S}$ and $\bar \H$
satisfy
\begin{align*}
  \partial_t \bar{S} +(\v \cdot \nabla) \bar{S} =0, \quad
   \partial_t \bar{\H}  + ( {\v}  \cdot \nabla) \bar{\H}
   - ( \bar{\H} \cdot \nabla) {\v}  =0
\end{align*}
in the sense of distributions. Since   $r(S^\epsilon,
\epsilon q^\epsilon)-r_0(S^\epsilon)=O(\epsilon)$, we have
\begin{align*}
(r(S^\epsilon, \epsilon
q^\epsilon)-r_0(S^\epsilon))(\partial_t
\u^\epsilon+(\u^\epsilon\cdot \nabla)\u^\epsilon)\rightarrow 0 ,
\end{align*}
whence,
\begin{align*}
r ( S^\epsilon, \epsilon q^\epsilon)(\partial_t
\u^\epsilon +(\u^\epsilon\cdot \nabla)\u^\epsilon )
 =\,& (r (S^\epsilon, \epsilon q^\epsilon)-r_0(S^\epsilon))(\partial_t
\u^\epsilon+(\u^\epsilon\cdot \nabla)\u^\epsilon)\nonumber\\
& +\partial_t(r_0(S^\epsilon)\u^\epsilon)+(\u^\epsilon \cdot \nabla)(r_0(S^\epsilon)\u^\epsilon)\\
\rightarrow &\, r_0(\bar S)(\partial_t \v +(\v\cdot \nabla)\v)
\end{align*}
in the sense of distributions.

Applying the operator \emph{curl} to
the momentum equations \eqref{nal} and then taking to the limit, we conclude that
\begin{align*}
\cu\big( r_0(\bar{S})(\partial_t \v+\v\cdot \nabla \v)
 -(\cu\bar{\H}) \times \bar{\H}-\mu\Delta \v \big)=0.
\end{align*}
Therefore, recalling $\cu\nabla =0$, we see that the limit
$(\bar S,\v ,\bar \H)$ satisfies
\begin{align}
&   r(\bar{S},0)(\partial_t \v+(\v\cdot \nabla) \v)
  -(\cu\bar{\H}) \times \bar{\H}-\mu\Delta \v +\nabla \pi =0,  \label{cba} \\
&  \partial_t \bar{\H}  + ( {\v}  \cdot \nabla) \bar{\H}
   - ( \bar{\H} \cdot \nabla) {\v} =0,  \label{cbb} \\
 &\partial_t \bar{S} +(\v \cdot \nabla) \bar{S} =0,  \label{cbc} \\
& \dv \v=0,  \quad \dv \bar{\H} =0\label{cbd}
\end{align}
for some function $\pi$.

If we employ the same arguments as in the proof
of \cite[Theorem 1.5]{MS01}, we find that $(\bar S, \v, \bar \H)$
satisfies the initial conditions \eqref{nax}. Moreover, the
standard iterative method shows that the system \eqref{cba}--\eqref{cbd}
with initial data \eqref{nax} has a unique solution
$(S^*, \v^*, \H^*)\in C([0,T],H^4(\mathbb{R}^d)).$ Thus, the uniqueness of
solutions to the limit system \eqref{cba}--\eqref{cbd} implies that
the above convergence results hold for
the full sequence $(S^\epsilon, q^\epsilon, \u^\epsilon,\H^\epsilon)$.
Thus, the proof is  completed.
\end{proof}

\medskip
 \noindent
{\bf Acknowledgements:}  The authors are grateful to the anonymous referees for their constructive
comments and helpful suggestions, which improved the earlier
version of this paper. The authors thank Professor
Fanghua Lin for suggesting this problem and helpful discussions.
This work was partially done when Li was visiting the Institute of
Applied Physics and Computational Mathematics in Beijing. He would
like to thank the institute for hospitality. Jiang was supported by
the National Basic Research Program under the Grant 2011CB309705
 and NSFC (Grant No. 11229101). Ju was supported by
NSFC (Grant No. 11171035). Li was supported  by NSFC (Grant No.
11271184, 10971094), PAPD, and the Fundamental Research Funds for the
Central Universities.
%--------------------------------------------------------------------

%\newpage

\end{document}